\documentclass[12pt]{amsart}
\usepackage{amsmath,amssymb,enumerate,mathtools,dsfont,bbding}
\newtheorem{theorem}{Theorem}[section]
\newtheorem{claim}{}[theorem]
\newtheorem{lemma}[theorem]{Lemma}

\newtheorem{conjecture}[theorem]{Conjecture}
\theoremstyle{definition}

\newtheorem{remark}[theorem]{Remark}

\newcommand{\bF}{\mathbb F}

\newcommand{\cC}{\mathcal{C}}

\newcommand{\cF}{\mathcal{F}}

\newcommand{\cH}{\mathcal{H}}

\newcommand{\cP}{\mathcal{P}}

\newcommand{\vs}[1]{\left[{#1}\right]}

\DeclareMathOperator{\cl}{cl}

\DeclareMathOperator{\PG}{PG}
\DeclareMathOperator{\PGS}{PGS}

\DeclareMathOperator{\AG}{AG}

\DeclareMathOperator{\Stab}{Stab}

\newcommand{\del}{ \backslash  }

\numberwithin{subcase}{case}
\numberwithin{subsubcase}{subcase}
\newenvironment{subproof}[1][\proofname]{%
  \begin{proof}[Subproof:]%
}{%
  \end{proof}%
}

\begin{document}
%\sloppy

\title{The smallest $I_5$-free and triangle-free binary matroids}
\author[Nelson]{Peter Nelson}
\address{Department of Combinatorics and Optimization, University of Waterloo, Waterloo, Canada. Email address: {\tt 	apnelson@uwaterloo.ca}}
\author[Nomoto]{Kazuhiro Nomoto}
\address{Department of Combinatorics and Optimization, University of Waterloo, Waterloo, Canada. Email address: {\tt 	knomoto@uwaterloo.ca}}
\thanks{This work was supported by a discovery grant from the Natural Sciences and Engineering Research Council of Canada and an Early Researcher Award from the government of Ontario}

\subjclass{05B35}
\keywords{matroids}
\date{\today}
\begin{abstract}
	We determine the smallest simple triangle-free binary matroids that have no five-element independent flat. This solves a special case of a conjecture of Nelson and Norin. 
\end{abstract}

\maketitle

\section{Introduction} This paper determines the smallest simple binary matroids that have no five-element independent flat and no triangle.

This problem is motivated by the following theorem by Nelson and Norin [\ref{nnorin}], which considers a similar extremal problem for matroids with no independent flat. Note that their theorem is valid not only for binary matroids, but also for general matroids. For integers $r \geq 1$ and $t \geq 1$, $M_{r,t}$ denotes the matroid that is the direct sum of $t$ (possibly empty) binary projective geometries whose ranks sum to $r$ and pairwise differ by at most $1$. 

\begin{theorem}[{[\ref{nnorin}], Theorem 1.1}]\label{noclaw}
Let $r, t \geq 1$ be integers. If $M$ is a simple rank-$r$ matroid with no independent flat of rank $(t+1)$, then $|M| \geq |M_{r,t}|$. If equality holds and $r \geq 2t$, then $M \cong M_{r,t}$. 
\end{theorem}

In the same paper [\ref{nnorin}], they made the following natural conjecture, in which it is required that the matroids be triangle-free. Our main result is the case $t=2$ in the setting of simple binary matroids. 

\begin{conjecture}[{[\ref{nnorin}], Conjecture 1.3}]\label{bigconjecture}
Let $t,r$ be integers with $t \geq 1$ and $t | r$. If $M$ is a simple triangle-free matroid with no rank-$(2t+1)$ independent flat, then $|M| \geq t2^{r/t-1}$. 
\end{conjecture}

The terminology in this paper is based on standard matroid theory, but it will be convenient to think of matroids as living inside an extrinsic ambient space, so we make the following minor modifications. A \emph{simple binary matroid} (which we simply refer to as just a \emph{matroid} in the rest of the paper) is a pair $M = (E,G)$, where $G$ is a finite binary projective geometry $\PG(n-1,2)$, and the \emph{ground set} $E$ is any subset of the points of $G$. For convenience, we write $G$ for the points of $G$. The \emph{dimension} of $M$ is the dimension of $G$ as a geometry. We write $|M|$ for $|E|$. Given two matroids $M_1 = (E_1, G_1)$, $M_2 = (E_2, G_2)$, we say that $M_1$ and $M_2$ are isomorphic if there exists an isomorphism $i: G_1 \rightarrow G_2$ such that $i(E_1) = E_2$. Having an explicit ambient space allows us to define the following. We say that a matroid $N$ is an \emph{induced restriction} (or $\emph{induced submatroid}$) of $M$ if there exists some subgeometry $F \subseteq G$ for which $N = (E \cap F, F)$. $M$ is \emph{$N$-free} if $M$ has no induced restriction isomorphic to $N$.

The ambient space of a matroid in this sense is analogous to the vertex set of a graph. In the same way that a graph might have isolated vertices, we do not require $E$ to span $G$. If $E$ spans $G$, then we say that $M$ is \emph{full-rank}; otherwise it is \emph{rank-deficient}. 

A \emph{triangle} is a two-dimensional subgeometry. Given $M=(E,G)$, if $E$ contains no triangle of $G$, then $M$ is \emph{triangle-free}. If $B$ is a basis of an $n$-dimensional projective geometry $G$, we write $I_n$ for the matroid $(B,G)$. When $n=3$, then we call $I_3$ a \emph{claw}. An \emph{affine geometry} of dimension $n$, denoted $\AG(n-1,2)$, is obtained by removing a maximal proper subgeometry from a projective geometry $\PG(n-1,2)$. 

We can now formulate and state our main result in this language. 

\begin{theorem}
Let $M=(E,G)$ be a full-rank, $I_5$-free, and triangle-free matroid of dimension $r$. Then $|E| \geq 2^{\lfloor r/2 \rfloor -1 } + 2^{\lceil r/2 \rceil -1 }$. When $r \geq 6$, if equality holds, then $M$ is the direct sum of two affine geometries of dimension $\lfloor r/2 \rfloor$ and $\lceil r/2 \rceil$ respectively.
\end{theorem}

The proof largely falls into two parts. The first step is to show that the tight examples do not contain $C_5$ as an induced restriction; $C_5$ is the $5$-element circuit. In the second part, we consider the class of $I_5$-free, $C_5$-free and triangle-free matroids to determine the tight examples.

\begin{remark}
The case $t=1$ in Conjecture \ref{bigconjecture} is trivial; it is not hard to show that the only matroids that are both $I_3$-free and triangle-free are affine geometries. 

Moreover, we remark that the analogous question for $I_4$-free and triangle-free matroids can also be answered. By using the structure theorem for $I_4$-free, triangle-free matroids [\ref{nn2}], it is straightforward to show that if $M$ is $I_4$-free and triangle-free, then $|E| \geq 2^{r-2}+1$. If equality holds, then either $M$ is the series extension of an affine geometry $\AG(r-2,2)$, or $\AG(r-2,2)$ with a coloop. Note that the tight examples are not unique for $I_4$-free and triangle-free matroids.
\end{remark}

\section{Preliminaries} 

\subsection*{Induced Restrictions} A subgeometry of $G$ is a \emph{flat} of $G$. Viewing $G$ as $\bF_n^2 \del \{0\}$, a set $F \subseteq G$ is a flat if and only if $F \cup \{0\}$ is closed under addition. Write $\vs{F} = F \cup \{0\}$. A two-dimensional flat is a \emph{triangle}, and a maximal proper flat of $G$ is a \emph{hyperplane}. Recall that a matroid $N$ is an \emph{induced restriction} (or equivalently $\emph{induced submatroid}$) of $M$ if there exists some flat $F \subseteq G$ for which $N = (E \cap F, F)$; we write $M | F$ for this matroid. 

In many of the proofs that follow, we obtain a contradiction by finding various induced restrictions. For the sake of brevity, we will sometimes say that a set $Z \subseteq E$ is an induced $N$-restriction if $M \rvert \cl(Z)$ is an induced $N$-restriction. We will also often simply assert that such a given set $Z$ is an induced $N$-restriction without performing all necessary checks explicitly. For example, when $N = I_n$, then in order to check that $Z = \{z_1, z_2, \dots, z_n\} \subseteq E$ is an induced $I_n$-restriction, one needs to check that $Z$ is independent, and $\sum_{i \in I} z_i \notin E$ for all $I \subseteq \{1,2,\dots,n\}$ and $|I| > 1$. We will not perform these calculations explicitly; instead we will provide enough information prior to such a claim so that these checks can be performed easily.

\subsection*{Contractions} Given a flat $F \subseteq G$, a \emph{coset} of $F$ is a set of the form $A = x + \vs{F}$ for some $x \in G \del F$. Note that the cosets of $F$ partition $G \del F$. We do not consider the set $F$ itself to be a coset. 

Given an $n$-dimensional matroid $M=(E,G)$ and a $k$-dimensional flat $F$ of $G$, let $F'$ be an $(n-k)$-dimensional flat of $G$ for which $F \cap F' = \varnothing$.  Then we define $M/F$, the \emph{contraction of $M$ by $F$}, to be the matroid $(E_F, F')$ where $E_F = \{x \in F' \mid E \cap (x + \vs{F}) \neq \varnothing\}$; note that such matroids are all isomorphic for any valid choice of $F'$. Every element of the ground set of the matroid $M/F$ is represented by a unique coset $A$ of $F$ for which $A \cap E \neq \varnothing$. We say that a coset $A$ is \emph{non-empty} if $A \cap E \neq \varnothing$. A coset is \emph{mixed} if $0 < |A \cap E| < |A|$. 

\subsection*{Affineness} A matroid $M = (E,G)$ is \emph{affine} if there exists a hyperplane $H \subseteq G$ for which $E \subseteq G \del H$. For $n \geq 1$, $C_n$ is the \emph{$n$-circuit}, the $(n-1)$-dimensional matroid whose ground set consists of $n$ points that add to zero. When $n$ is odd, we say that it is an \emph{odd circuit}. When $t=3$, then $C_3$ is just a triangle. The following is the matroidal analogue of the standard characterisation of bipartite graphs in terms of excluded odd cycles. The proof is standard, and is also included in [\ref{nn2}].

\begin{theorem}
A matroid $M = (E,G)$ is affine if and only if $M$ has no induced odd circuits. 
\end{theorem}

Observe that odd circuits of length $6$ or more contain an induced $I_5$-restriction. Hence, if $M$ is an $I_5$-free, triangle-free matroid, then $M$ is affine if and only if there exists no induced $C_5$-restriction of $M$.

\subsection*{Doublings} Let $M=(E,G)$ be a matroid, and let $H$ be a hyperplane of $G$, and $a \in G \del H$. Then $M$ is the \emph{doubling of $M \rvert H$ by $a$} if $E = \vs{a} + (E \cap H)$. If this condition holds, then in fact $M$ is the doubling of $M \rvert H'$ for any hyperplane $H'$ not containing $a$. 

It is easy to see that any two doublings of the same matroid $N$ are isomorphic. The element $a$ is the \emph{apex}. We sometimes omit the element $a$, and simply say that $M$ is the \emph{doubling of $N$}, or even say that $M$ is a doubling if such $a$ and $N$ exist. We write $D(N)$ to denote the doubling of $N$, and define $D^k(N) = D(D^{k-1}(N))$ recursively for $k > 1$. We say that $D^k(N)$ is the \emph{$k$-doubling of $N$}, or \emph{the doubling of $N$ of order $k$}. The apex in the case of a $k$-doubling is a $k$-dimensional flat $D$ for which $E = \vs{D} + (E \cap F)$ where $M \rvert F = N$. 

\subsection*{Claw-freeness}
In this paper, we will often resort to the properties of claw-free matroids, as these matroids are well-understood. In fact, there exists a full structure theorem for all claw-free matroids [\ref{nn1}], although we do not need it in this paper. Here, we collect some facts about claw-free matroids that we will use in this paper.

\begin{lemma}[{[\ref{bkknp}]}]\label{i3triangle}
If $M$ is a full-rank, $I_3$-free, triangle-free matroid, then $M$ is an affine geometry.
\end{lemma}
 
The next is in fact a special case of Theorem \ref{noclaw} proved by Nelson and Norin [\ref{nnorin}]. The following special case for claw-free matroids first appeared in [\ref{nn1}].

\begin{theorem}[{[\ref{nn1}], [\ref{nnorin}]}]\label{I3free_extremal}
Let $M = (E,G)$ be an $r$-dimensional, $I_3$-free matroid. Then $|E| \geq 2^{\lfloor r/2 \rfloor} + 2^{\lceil r/2 \rceil}-2$.
\end{theorem}

\section{The none-affineness of extremal examples}
In this section our goal is to show that the smallest $I_5$-free and triangle-free matroids are affine when $r \geq 6$. The strategy is to start with an induced $C_5$-restriction and show that such matroids always contain too many elements. The argument hinges on the following observation, which will allow us to transform our problem to that of $I_3$-free matroids, which we understand well.

\begin{lemma}\label{c5contraction}
Let $M = (E,G)$ be a full-rank, $I_5$-free and triangle-free matroid. Let $F$ be a flat of $G$ so that $M | F \cong C_5$. Then $M / F$ is $I_3$-free. 
\end{lemma}

\begin{proof}
Suppose for a contradiction that $M / F$ contains an induced $I_3$-restriction. Let $A_1, A_2, A_3$ be the three non-empty cosets of $F$ that correspond to this induced $I_3$-restriction.  Fix $a_k \in A_k \cap E$. Denote the five elements of $M | F$ by $\{x_1,x_2, x_3, x_4, x_5\}$, and consider the colouring of the complete graph on $5$ vertices labelled $V = \{v_1, v_2, v_3, v_4, v_5\}$. We colour the edge $v_i v_j$ by colour $k$ if $a_k + x_i + x_j \in E$ (so an edge can have more than one colour at the same time). We now make the following sequence of observations.

\begin{claim}\label{somecolour}
Every edge is coloured with at least one colour.
\end{claim}

\begin{subproof}
If there exists an edge $e = v_i v_j$ that is not coloured at all, then $\{a_1, a_2, a_3, x_i, x_j\}$ is an induced $I_5$-restriction, a contradiction.
\end{subproof}

\begin{claim}\label{non_adj_colour}
If an edge is coloured with colour $k$, then the non-adjacent edges are not coloured with $k$.
\end{claim}

\begin{subproof}
Suppose otherwise. Without loss of generality, suppose that the edge $e=v_1 v_2$, is coloured $k$, and that there exists a non-adjacent edge $e' = v_3 v_4$ that is coloured $k$. Then $\{a_k + x_1 + x_2, a_k + x_3 + x_4, x_5 \}$ is a triangle, a contradiction.
\end{subproof}

The above statement, \ref{non_adj_colour}, implies that each colour class is either a triangle or a star in the graph. 

\begin{claim}\label{non_mono_triangle}
There is no triangle whose edges share the same colour.
\end{claim}

\begin{subproof}
Suppose for a contradiction that there is a triangle, say $\{v_3 v_4, v_4 v_5, v_3 v_5\}$ coloured with colour $3$. By \ref{somecolour} and \ref{non_adj_colour}, the edge $v_1 v_2$ has at least one of the colours $1$, $2$. Suppose without loss of generality that it is coloured with colour $1$ (and potentially with colour $2$ as well). 

Now by \ref{non_adj_colour} again, we can conclude that each one of the edges $\{v_1v_3,v_1v_4, v_1v_5, v_2v_3,v_2v_4,v_2v_5\}$ cannot be coloured with $3$. Moreover, by using \ref{somecolour} and \ref{non_adj_colour}, each of them is coloured with exactly one of $1$ or $2$; if, without loss of generality, $v_1v_3$ was coloured $1$ and $2$, and then $v_2v_4$ would have no available colours, a contradiction. 

Now, we claim that $\{v_1v_3,v_1v_4, v_1v_5\}$ share the same colour, and $\{v_2v_3,v_2v_4, v_2v_5\}$ share the other colour. Suppose that $v_1v_3$ has colour $1$. Then by \ref{non_adj_colour}, $v_2v_4$ and $v_2v_5$ have colour $2$. By \ref{non_adj_colour}, this then implies that $v_1v_4$ and $v_1v_5$ have colour $1$. Finally, by \ref{non_adj_colour} again, this implies that $v_2v_3$ has colour $2$. 

But then $\{x_1, x_3, x_1+x_2+a_1, a_2, a_3\}$ is an induced $I_5$-restriction, a contradiction.
\end{subproof}

Now we can finish the proof. By \ref{non_adj_colour}, we have that each colour class is either a triangle or a star in the graph. By \ref{non_mono_triangle}, we have that each colour class is a star. Let $c_k$ be the centre of the star corresponding to the colour $k$. Since $|V \del \{c_1, c_2, c_3\}| \geq 2$, pick two distinct vertices $v_i, v_j \in V \del \{c_1, c_2, c_3\}$. But then the edge $v_i v_j$ is not coloured with any colour, contradicting \ref{somecolour}.
\end{proof}

We will make one more observation about the cosets of $F$.  

\begin{lemma}\label{apex_of_1}
Let $A, B, C$ be three cosets of $F$ where $M \rvert F \cong C_5$ for which $|A \cap E|=1$ and $B+ C= A$, $|B \cap E| \leq |C \cap E|$ and $|C \cap E| \neq 0$. If $|B \cap E| \leq 2$, then $|C \cap E| \geq 3$.
\end{lemma}

\begin{proof}Suppose for a contradiction that $|B \cap E| \leq 2$ but $|C \cap E| \leq 2$. 

We first consider the case $|B \cap E| = 0$. Fix $a_1 \in A \cap E$ and $c_1 \in C \cap E$. Then $|C \cap E| =2$; otherwise $\{a_1, c_1, x_1, x_2, x_3\}$ for any three elements $x_1, x_2, x_3 \in E \cap F$ gives an induced $I_5$-restriction. Let $c_2 \in C \cap E$. Since $M$ is triangle-free, we have that $c_1 + c_2 = x_1+x_2$ for some two $x_1, x_2 \in E \cap F$. Pick any other two elements $x_3, x_4 \in E \cap F$. Then $\{a_1, c_1, x_1, x_3, x_4\}$ is an induced $I_5$-restriction, a contradiction. 

Hence we may assume that $|B \cap E| > 0$. Now fix $b_1 \in B \cap E$ and $c_1 \in C \cap E$. Denote the five elements of $M | F$ by $\{x_1,x_2, x_3, x_4, x_5\}$, and consider the colouring of the complete graph on $5$ vertices labelled $V = \{v_1, v_2, v_3, v_4, v_5\}$. We colour the edge $v_i v_j$ by colour $b$ if $b_1 + x_i + x_j \in E$, and $c$ if $c_1 + x_i + x_j \in E$. Since $0 < |B \cap E|, |C \cap E| \leq 2$, there is at most one edge of colour $b$ and $c$. We colour the vertex $v_i$ with colour $a$ if $b_1+c_1 + x_i \in E$, and the edge $v_i v_j$ by colour $a$ if $b_1+c_1 + x_i + x_j \in E$. Note that there is precisely one object (vertex or edge) that is coloured $a$. We make the following claims. 

\begin{claim}
If an edge is coloured $b$ or $c$, then it cannot be coloured $a$. 
\end{claim}

\begin{subproof}
Suppose not, and assume without loss of generality that the edge $v_1 v_2$ is coloured with $a$, one of $b$ and $c$. If the extra colour if $b$, then $\{b_1+c_1+x_1+x_2, b_1+x_1+x_2, c_1\}$ is a triangle, a contradiction. The case the extra colour is $c$ is symmetrical. 
\end{subproof}

\begin{claim}
If two non-adjacent edges $v_iv_j$ and $v_k v_l$ are coloured with $b$ and $c$ respectively, then the vertex $v_t$ cannot be coloured $a$ where $v_t \in V \del \{v_i, v_j, v_l, v_k\}$.
\end{claim}

\begin{subproof}
If not, then $\{b_1+c_1+x_t, b_1+x_i+x_j, c_1+x_k + x_l\}$ is a triangle. 
\end{subproof}

\begin{claim}
There are no three uncoloured vertices $v_i, v_j, v_l$ such that the edges $v_i v_j$, $v_i v_l$, $v_j v_l$ and $v_k v_l$ are uncoloured, where $\{v_k,v_l\} = V \del \{v_i, v_j, v_l\}$.  
\end{claim}

\begin{subproof}
Suppose for a contradiction that such vertices exist. Then $\{b_1, c_1, x_i, x_j, x_l\}$ is an induced $I_5$-restriction. 
\end{subproof}

\begin{claim}
There are no triangles in the graph where each edge is coloured with a distinct colour.  
\end{claim}

\begin{subproof}
Suppose not, so that, without loss of generality, $v_1 v_2$ is coloured $a$, $v_2 v_3$ is coloured $b$, $v_1 v_3$ is coloured $c$. Then $\{b_1+c_1+x_1+x_2, b_1+x_2+x_3, c_1+x_1+x_3\}$ is a triangle.
\end{subproof}

Through a case analysis, we find that there is no valid colouring of the vertices and edges of the complete graph under these constraints. The case analysis is omitted as it is routine.
\end{proof}

Combining Lemma~\ref{c5contraction}, Lemma~\ref{apex_of_1} and Theorem~\ref{I3free_extremal} yields the following result, which shows that the smallest $I_5$-free, triangle-free matroids are affine. 

\begin{lemma}\label{extremal_examples_are_affine}
Let $M=(E,G)$ be an $r$-dimensional, full-rank, $I_5$-free and triangle-free matroid with an induced $C_5$-restriction on a flat $F$, so that $M | F \cong C_5$. Then, $|E| \geq 2^{\lfloor r/2 \rfloor -1 } + 2^{\lceil r/2 \rceil -1 }$. Moreover, when $r \geq 6$, this is a strict inequality.
\end{lemma}

\begin{proof}
By Lemma~\ref{c5contraction}, $M / F$ is $I_3$-free. It is also full-rank; otherwise, $M$ is rank-deficient. By Theorem~\ref{I3free_extremal}, this means $|M / F| \geq 2^{\lfloor r-4/2 \rfloor} + 2^{\lceil r-4/2 \rceil}-2$. We may assume that there exists some non-empty coset of $F$ that contains exactly $1$ element; otherwise, 
\begin{align*}
	|E| \geq |M | F| + 2 |M/F| &= 5+ 2(2^{\lfloor r-4/2 \rfloor} + 2^{\lceil r-4/2 \rceil}-2)  \\
	&= 1 + 2^{\lfloor r/2 \rfloor -1 } + 2^{\lceil r/2 \rceil -1 } \\
	&> 2^{\lfloor r/2 \rfloor -1 } + 2^{\lceil r/2 \rceil -1 }.
\end{align*}
which already satisfies the lemma.

Let $A$ be a coset of $F$ for which $|A \cap E|=1$. Let $\cC$ be the set of cosets of $F$. We will partition the cosets of $F$ as follows by $T_1$ and $T_0$. 
$$T_ 1 = \{B \mid B \neq A, |B \cap E| > 0, |(A+B) \cap E| > 0 \}$$
$$T_0 = \{B \mid B \neq A, |B \cap E| > 0, |(A+B) \cap E| = 0\}.$$
By Lemma~\ref{apex_of_1} , we have that if $B \in T_1$, then $|B \cap E| + |(A+B) \cap E| \geq 4$. Therefore

$$ \sum_{B  \in T_1} |B \cap E|  = \frac{1}{2} \sum_{B  \in T_1}( |B \cap E| + |(A+B) \cap E|)  \geq  2 |T_1|.$$

By Lemma~\ref{apex_of_1} again, for every $B \in T_0$, we have that $|B \cap E| \geq 3$. 

$$\sum_{B \in T_0} |B  \cap E|  \geq 3 |T_0|$$

Therefore we have

	\begin{align*}
		|E|  &= |M | F| + \sum_{B \in \cC} |B \cap E|  \\
				 &= |M | F| + |A \cap E| +  \sum_{B \in T_1} | B \cap E| + \sum_{B \in T_0}|B \cap E| \\
				 & \geq 5 + 1 + 2 |T_1| + 3 |T_0| \\
				 & \geq 6 + 2 (|T_1|+ |T_0|) \\
				 & = 6 + 2 (|M / F| - 1) \\
				 & \geq 6 + 2 (2^{\lfloor (r-4)/2 \rfloor} + 2^{\lceil (r-4)/2 \rceil}-3) \\
				 & = 2^{\lfloor r/2 \rfloor -1 } + 2^{\lceil r/2 \rceil - 1},
	\end{align*}

This proves the bound. Moreover, note that equality can only be achieved when $T_0$=0 and $|M / F| = 2^{\lfloor (r-4)/2 \rfloor} + 2^{\lceil (r-4)/2 \rceil}-2$, which implies $|T_1| = 2^{\lfloor (r-4)/2 \rfloor} + 2^{\lceil (r-4)/2 \rceil}-3$. But note that $T_1$ is always even by definition, which is a contradiction, unless $r=5$. This completes the proof.
\end{proof}

\section{The contraction of a maximum affine geometry}
In the previous section, we showed that the smallest matroids with no $I_5$-restriction or triangle must be affine when $r \geq 6$. This means that such matroids are $C_5$-free. The goal of this section therefore is to study the matroids that are $I_5$-free, triangle-free and $C_5$-free. The trick is to start with a maximum affine geometry of $M$ and consider its contraction. 

We divide this section into three subsections. Suppose this maximum affine geometry is induced on a flat $F$. In the first subsection, we show that $M / F$ is close to being $I_3$-free, in the sense that the only obstruction is a matroid we refer to as a \emph{doubled kite of order $k$}. In the second subsection, we show that $M / F$ is always $2T$-free ($2T$ is a six-dimensional matroid consisting of two disjoint triangles). In the third subsection, we determine the smallest $I_3$-free and $2T$-free matroids. Combining these pieces in the next section will yield our result.

\subsection{$I_3$-freeness and kites}

\subsection*{Doubled kites} A \emph{kite} is a $6$-dimensional matroid $(E,G)$ with $10$ elements with a basis $B=\{x_1,x_2,x_3,x_4,x_5,x_6\}$ such that $E = B \cup \{x_1+x_2+x_3, x_2+x_3+x_4, x_1+x_3+x_5,x_1+x_2+x_6\}$. A \emph{doubled kite of order $k$} is a $(k+6)$-dimensional matroid obtained by a sequence of $k$ doublings of a kite. When we do not wish to specify the number of doublings, we simply say that it is a \emph{doubled kite}. 

\begin{lemma}\label{contraction_max_ag}
Let $M | F$ be an affine geometry of an $I_5$-free, triangle-free and $C_5$-free matroid of dimension at least $3$. Suppose $M / F$ contains an induced $I_3$-restriction, represented by three non-empty cosets $A_1, A_2,A_3$ of $F$ in $M$. Then either $M \rvert \cl(F \cup A_i)$ is an affine geometry for some $i$, or $M | \cl(F \cup A_1 \cup A_2 \cup A_3)$ is a doubled kite of order $\dim(F)-3$.
\end{lemma}

\begin{proof}
Let $W$ be the hyperplane of $F$ for which $E \cap W = \varnothing$. Suppose that $M \rvert \cl(F \cup F_i)$ is not an affine geometry for any $i=1,2,3$. For each $i$, let $\Stab(A_i \cap E) = \{w \in W | w + A_i \cap E= A_i \cap E\}$. Note that $\Stab(A_i \cap E)$ is a flat of $W$. 

\begin{claim}\label{stab_covering}
$W = \bigcup\limits_{i=1}^{3} \Stab(A_i \cap E)$.
\end{claim}

\begin{subproof}
Let $w \in W$. If $w \notin \cup_{i=1}^{3} \Stab(A_i\cap E)$, then there exists $a_i \in A_i \cap E$ for $i=1,2,3$ for which $w + a_i \notin A_i \cap E$. Pick any $z \in E \cap F$. Then $\{z, z+w, a_1, a_2, a_3\}$ is an induced $I_5$-restriction.  
\end{subproof}

\begin{claim}
For $i=1,2,3$, $\Stab(A_i \cap E)$ is a hyperplane of $W$. Moreover, the three hyperplanes intersect at a codimension-$2$ flat of $W$.
\end{claim}

\begin{subproof}
Since $M \rvert \cl(F \cup F_i)$ is not an affine geometry for any $i=1,2,3$, it follows that $\Stab(A_i \cap E)$ is a strict subspace of $W$. Let $W_i = \Stab(A_i \cap E)$.

Suppose for a contradiction that the $W_i$'s are not all distinct. Without loss of generality assume that $W_1=W_2$. Then
$$|\bigcup\limits_{i=1}^{3} \Stab(A_i \cap E)| \leq 2^{\dim(W_1)} + 2^{\dim(W_3)}-2 \leq 2 \cdot 2^{\dim(W)-1}-2 < |W|$$
This is a contradiction to ~\ref{stab_covering}. Hence the $W_i$'s are all distinct. 

We may assume that $\dim(W) > 2$, as otherwise the statement now follows trivially; if $\dim(W)=2$, then since $W_i$ is a strict subspace of $W$, it follows that $|W_i| < 2$. By \ref{stab_covering}, $|W_i| = 1$ for $i=1,2,3$ and are disjoint. 

Note that we can also use similar size considerations to deduce that the $W_i$ are hyperplanes. Suppose for a contradiction that $W_1$ is not a hyperplane. Since $W_2$ and $W_3$ do not equal $W$, $|W_2 \cup W_3|$ is at most the size of the union of two hyperplanes, so $|W_2 \cup W_3| \leq 2^{\dim(W)-2}-1+2^{\dim(W)-1}$. But then $|W_1 \cup W_2 \cup W_3| \leq |W_1| + |W_2 \cup W_3| \leq 2^{\dim(W)-2} - 1 + 2^{\dim(W)-2}-1+2^{\dim(W)-1} = 2^{\dim(W)}-2 < |W|$, a contradiction. 

Now, suppose that $W_1$ and $W_2$ intersect at a codimension-$2$ flat $W'$ of $W$. Then $W_3$ either contains $W'$, in which case we are done, or $\dim(W_3 \cap W') = \dim(W)-3$. If $\dim(W_3 \cap W') = \dim(W)-3$, let $W'' = W_3 \cap W'$, and $W_i'' = W_i \cap W_3 \del W''$ for $i=1,2$. But we may take $w_i \in W_i \del (W' \cup W_i'')$ for $i=1,2$, and $z \in W' \del W''$. But then $w_1+w_2  + z\notin W_1 \cup W_2 \cup W_3$, a contradiction to ~\ref{stab_covering}. 
\end{subproof}

Now we claim that $M | \cl(F \cup A_1 \cup A_2 \cup A_3)$ is a doubled kite. Let $W' = W_1 \cap W_2 \cap W_3$. By the above, $W'$ has dimension $\dim(W)-2$. Pick a triangle $T \subseteq W \del W'$, and fix $a_i \in A_i \cap E$ for $i=1,2,3$, and $x \in F \del W$. Then $M \rvert \cl(T \cup \{a_1,a_2,a_3,x\})$ is a kite, which is doubled $\dim(W'')$ times to give $M$. This completes the proof.
\end{proof}

Lemma \ref{contraction_max_ag} tells us that if we start with an affine geometry $M \rvert F$ (with a hyperplane $W$ for which $F \cap W = \varnothing$) whose contraction gives an induced $I_3$-restriction, then either we can obtain a larger affine geometry, or we obtain a doubled kite. Moreover, when we obtain a doubled kite, then, for any three elements $a_1, a_2, a_3 \in E$ such that $\cl(a_1,a_2,a_3)$ has dimension $3$ and is disjoint from $F$, the doubled kite has the form $(F \del W) \cup \bigcup_{i=1}^{3}(a_i + \vs{W_i})$ where $W_i$ are distinct hyperplanes of $W$ with the property that $W_1 \cap W_2 \cap W_3$ is a codimension-$2$ flat of $W$.

In the proof of the main theorem, we will take a largest affine geometry, and then contract it. Our hope is that the contraction is $I_3$-free, as we understand $I_3$-free matroids well. But Lemma~\ref{contraction_max_ag} implies that such the contraction of a largest affine geometry could lead to a doubled kite. It turns out, however, that the existence of a doubled kite is enough to conclude that the matroid contains too many elements.

Note that a doubled kite of order $k$ contains the $(k+1)$-doubling of a claw, and the $k$-doubling of an induced $C_6$-restriction.

\begin{lemma}\label{kites_not_extremal}
For an $I_5$-free, $C_5$-free and triangle-free matroid $M = (E,G)$, let $k+3$ be the dimension of a maximum affine geometry contained in $M$ where $k \geq 0$. If $M$ contains an induced $D^{k+1}(I_3)$-restriction, then its contraction is $I_3$-free.
\end{lemma}

\begin{proof}

Let $D$ be a $(k+1)$-dimensional flat of $G$ and let $F$ be a flat such that $M | F \cong I_3$ and $E \cap \cl(D \cup F) = \vs{D} + (E \cap F)$. Let $F' =\cl (D \cup F)$, so that $M \rvert F'$ is a $(k+1)$-doubling of a claw. Write $F \cap E = \{x_1, x_2, x_3\}$.

Suppose for a contradiction that $M/F'$ contains a claw, and pick an element $y_i \in E$ from each of these three cosets corresponding to the claw. For $i=1,2,3$, let $U_i = \{v \in F' \del E \mid y_i + v \in E \}$. 

\begin{claim}
$U_i \cap (x_1+x_2+x_3 + \vs{D}) = \varnothing$, for $i=1,2,3$.
\end{claim}

\begin{subproof}
If there exists $v \in U_i \cap (x_1+x_2+x_3 + \vs{D})$, then $\{y_i, v+y_i, x_1, x_2, v +x_1+x_2\}$ is an induced $C_5$-restriction. 
\end{subproof}

\begin{claim}\label{doubled_claw_all_colour_labels}
$x_1+x_2 + \vs{D} \subseteq \bigcap_{i=1}^{3}U_i$
\end{claim}

\begin{subproof}
Note that $\cl(D \cup \{x_1, x_2\})$ is a largest affine geometry of $M$ since it has dimension $k+3$. Hence by Lemma \ref{contraction_max_ag}, $M | \cl(D \cup \{x_1, x_2, y_1, y_2, y_3\})$ is a doubled kite. Hence for every point $v \in x_1+x_2 + \vs{D}$, either there exists a unique $l \in \{1,2,3\}$ for which $v \in U_l$, or $v \in U_i$ for every $i = 1,2,3$. If there exists a unique $l \in \{1,2,3\}$ for which $v \in U_l$, then $\{y_1, y_2, y_3, v + y_l, x_3\}$ is an induced $I_5$-restriction, a contradiction. Hence, $x_1+x_2 + \vs{D} \subseteq \bigcap_{1}^{3}U_i$.
\end{subproof}

This is enough to derive a contradiction. Note again that $M | \cl(D \cup \{x_1, x_2, y_1, y_2, y_3\})$ is a doubled kite. This implies in particular that $J =\cl(D \cup \{x_1, x_2\}) \cap U_1 \cap U_2 \cap U_3$ is a flat of codimension $2$ of $\cl(D \cup \{x_1+x_2\})$; note $|J| = 2^{\dim(D)-1}-1$. By \ref{doubled_claw_all_colour_labels}, we have that $x_1+x_2 + \vs{D} \subseteq J$. But $|(x_1+x_2 + \vs{D})| = 2^{\dim(D)}$, a contradiction.
\end{proof}

The next goal is to bound the number of elements in a non-empty coset of a doubled $C_6$-restriction of order $k$. Let $d_k$ denote the minimum number of elements in a non-empty coset of a doubled $C_6$-restriction of order $k$. We will provide a lower bound on $d_k$ in two steps. In the first step we will show that $d_0 \geq 4$. 

\begin{lemma}\label{kite_hyperplane}
$d_0 \geq 4$.
\end{lemma}

\begin{proof}
Suppose $M = (E,G)$ is a $7$-dimensional matroid with a hyperplane $H \subseteq G$ for which $M | H \cong C_6$. Suppose $|E \del H| > 0$. Our aim is to show that $|E \del H| \geq 4$. Fix $v \in E \del H$. We note that $|E \del H| \geq 2$; otherwise $\{v, v_1, v_2, v_3, v_4\}$ is an induced $I_5$-restriction for any four elements $v_i \in E \cap H$, $i=1,2,3,4$. Because of $C_5$-freeness, $v + v_1+v_2+v_3 \notin E$ for any three distinct $v_1, v_2, v_3 \in E \cap H$. Hence we have some element $v' \in E \del H$ of the form $v' = v + v_1+ v_2$ for some two $v_1, v_2 \in E \cap H$. 

Now, consider the flat $F = \cl(v_1,v_3,v_4,v_5)$ for any three $v_3,v_4,v_5 \in E \cap H$ other than $v_2$. Now note that $M | \cl(F \cup \{v\})$ is an induced $I_5$-restriction unless $|E \cap (v + F)| \neq 0$, and similarly $M | \cl(F \cup \{v'\})$ is an induced $I_5$-restriction unless $|E \cap (v' + F)| \neq 0$. But $(v + F') \cap (v' + F) = \varnothing$. Hence $|E \del H| \geq 4$.
\end{proof}

\begin{lemma}\label{doubled_kite_hyperplane}
$d_k \geq 3 \cdot 2^{k} + 1$
\end{lemma}

This will follow directly from the following by setting $\cP$ to be the family of matroids that are $I_5$-free, $C_5$-free and triangle-free, and $N = C_6$ and Lemma \ref{kite_hyperplane}. A matroid property $\cP$ is \emph{hereditary} if it is inherited under taking induced restrictions.

\begin{lemma}
Let $\cP$ be a hereditary matroid property. Suppose that $N \in \cP$ and, for each $M = (E,G) \in \cP$ having a hyperplane $H$ with $M | H \cong N$ and $|E \del H| > 0$, we have $|E \del H| > t$. Then, if $M = (E,G) \in \cP$ has $D^{k}(N)$ on a hyperplane $H$ with $|E \del H| > 0$, then $|E \del H| > t \cdot 2^{k}$.
\end{lemma}

\begin{proof}
Suppose $M = (E,G)$ is such that $M$ has an induced $D^{k}(N)$-restriction on a hyperplane $H$. Note that given an $k$-dimensional flat, the number of $l$-dimensional subspaces of an $n$-dimensional space avoiding that $k$-dimensional flat is $2^{kl} {n-k \choose l}_2$ (this can be proved by a standard counting argument). 

Let $D$ be the $k$-dimensional flat apex that gives rise to the $k$-doubling of $N$. Let $\cF$ be the set of $dim(N)$-dimensional flats $F$ of $H$ such that $F \cap D = \varnothing$ (and hence $M | F \cong N$). Then $|\cF| = 2^{k \cdot \dim(N)}$.

Fix $v \in E \del H$. Let $S = \{(F,e) \mid F \in \cF, e \in (F + v) \cap E \}$. By assumption, we know that $|S| \geq t |\cF|$. We can also count this quantity in a different way, by first fixing $e \in ((H \del D) + v) \cap E$. Observe that given a point $y \in H \del D$, the number of $\dim(N)$-dimensional flats $F \in \cF$ that contain $y$ is $2^{k(\dim(N)-1)}$. Therefore, we have the following.

	\begin{align*}
		|S| &= \sum_{e \in  ((H \del D) + v) \cap E} |\{F \in \cF \mid e \in F+v\}| \\
		& = \sum_{e \in  ((H \del D) + v) \cap E} 2^{k(\dim(N)-1)}.
	\end{align*}
	
Hence it follows that $|((H \del D) + v) \cap E|  \cdot 2^{k(\dim(N)-1)} \geq t \cdot 2^{k \cdot \dim(N)}$. Therefore, $|(H + v) \cap E| \geq t \cdot 2^k$. Adding the fixed point $v$, we obtain that $|E \del H| > t \cdot 2^k $. 
\end{proof}

We now combine Lemmas \ref{kites_not_extremal} and \ref{doubled_kite_hyperplane}

\begin{lemma}\label{kite_too_many_elements}
Let $M=(E,G)$ be a full-rank, $r$-dimensional, $I_5$-free, $C_5$-free and triangle-free matroid. Let $k+3$ be the dimension of a maximum affine geometry contained in $M$ where $k \geq 0$, and suppose that $M$ contains the $k$-doubling of a kite as an induced restriction. Then $|E| > 2^{\lfloor r/2 \rfloor -1 } + 2^{\lceil r/2 \rceil -1 }$.
\end{lemma}

\begin{proof}
Let $F$ be a flat for which $M | F$ is the $k$-doubling of a kite, which contains an induced $D^{k+1}(I_3)$-restriction. Let $M_F = (E_F,F')$ be the contraction of $M$ by $F$. By Lemma \ref{kites_not_extremal}, $M_F$ is $I_3$-free. Also, $M_F$ is full-rank, as otherwise $M$ is rank-deficient. By Theorem \ref{I3free_extremal}, $|E_F| \geq 2^{\lfloor (r-k-6)/2 \rfloor} + 2^{\lceil (r-k-6)/2 \rceil}-2$. 

Note that $M|F$ contains an induced $D^k(C_6)$-restriction. By Lemma \ref{doubled_kite_hyperplane}, the coset corresponding to an element of $E_F$ contains at least $3 \cdot 2^{k}+1$ elements of $E$. Note $3 \cdot 2^{k}+1 \geq 4 \cdot 2^{k/2}$ for $k \geq 0$. Hence

	\begin{align*}
		|E|  &\geq  |E \cap F| + |E_F|(4 \cdot 2^{k/2})   \\
		&\geq 10 \cdot 2^k + (2^{\lfloor (r-k-6)/2 \rfloor} + 2^{\lceil (r-k-6)/2 \rceil}-2)(4 \cdot 2^{k/2}) \\
		& = 10 \cdot 2^k - 8 \cdot 2^{k/2} + 2^{\lfloor r/2 \rfloor-1} + 2^{\lceil r/2 \rceil -1} \\
		& > 2^{\lfloor r/2 \rfloor-1} + 2^{\lceil r/2 \rceil -1}.
	\end{align*}
\end{proof}

\subsection{$2T$-freeness}
We consider another matroid which we refer to as $2T$. The matroid $2T$ is the $6$-dimensional matroid with $6$ elements which is the direct sum of two triangles.

\begin{lemma}\label{2tfreeness}
Let $M | F$ be a maximal affine geometry of an $I_5$-free, triangle-free and $C_5$-free matroid, where $\dim(F) \geq 3$. Then $M / F$ is $2T$-free, or $M$ contains an affine geometry of dimension $\dim(F)+1$.
\end{lemma}

\begin{proof}
Let $W$ be the hyperplane of $F$ for which $W \cap E = \varnothing$.

Suppose for a contradiction that $M / F$ contains an induced $2T$-restriction. We think of this induced $2T$-restriction in terms of $6$ distinct non-empty cosets of $F$ defined on $A_1, A_2, A_3, B_1, B_2, B_3$ such that $A_1 + A_2 = A_3$ and $B_1 + B_2 = B_3$. 

We now make the following sequence of claims to uncover the structure of $M$. 

\begin{claim}\label{2t_c5}
For any distinct $i,j \in \{1,2,3\}$, if $a_i \in A_i \cap E$, $a_j \in A_j \cap E$, then $(a_i + a_j + \vs{W}) \cap E = \varnothing$.
\end{claim}
\begin{subproof}
Since $M$ is triangle-free, $a_i+a_j \notin E$. If there exists $a_k \in (a_i + a_j + W) \cap E$, then $\{a_i, a_j, a_k, z, a_i+a_j+a_k+z\}$ is an induced $C_5$-restriction for any $z \in F \cap E$, a contradiction. 
\end{subproof}

\begin{claim}\label{2t_base}
There exists $y \in F \cap E$ such that for any $\{i,j,k\} = \{1,2,3\}$, $A_i \cap E + A_j\cap E + y \subseteq A_k \cap E$. 
\end{claim}

\begin{subproof}
Pick $b_1 \in B_1 \cap E$ and $b_2 \in B_2 \cap E$. As $M$ is triangle-free, $b_1+b_2 \notin B_3 \cap E$. Due to maximality, we may pick $y \in F \cap E$ such that $b_1+b_2+y \notin E$; if not, then $M | \cl(F \cup \{b_1+b_2\})$ is a strictly larger affine geometry containing $M | F$. 

Let $a_i \in A_i \cap E$ and $a_j  \in A_j \cap E$. Then $y + a_i + a_j \in A_k \cap E $; otherwise $\{y, a_i, a_j, b_1, b_2\}$ is an induced $I_5$-restriction. 
\end{subproof}

\begin{claim}\label{2t_size}
For distinct $i,j \in \{1,2,3\}$, $|A_i \cap E + A_j\cap E| = |A_i\cap E|$.
\end{claim}

\begin{subproof}
By \ref{2t_base}, we have that $|A_i \cap E + A_j\cap E| \leq |A_k \cap E|$. Also, $|A_k \cap E| \leq |A_k \cap E + A_j \cap E|$. By applying \ref{2t_base} again, we have that $|A_k \cap E + A_j \cap E| \leq |A_i \cap E|$. Hence $|A_i \cap E + A_j \cap E| \leq |A_i \cap E|$. The claim follows. 
\end{subproof}

Note that \ref{2t_size} implies in particular that $|A_i \cap E| = |A_j \cap E|$.

Let $\Stab(A \cap E) = \{w \in \vs{W}| w + A  \cap E= A \cap E\}$. $\Stab(A \cap E)$ is a flat. Note that trivially $|A \cap E| \geq |\Stab(A \cap E)|$ when $|A \cap E| > 0$. We now claim the following.

\begin{claim}\label{2t_stab}\
\begin{itemize}
	\item For $i=1,2,3$, $|A_i \cap E| = |\Stab(A_i \cap E)|$,
	\item $\Stab(A_1 \cap E) = \Stab(A_2 \cap E) = \Stab(A_3 \cap E)$.
\end{itemize}
\end{claim}

\begin{subproof}
Fix $i$, and take some other $j \neq i$. 
By \ref{2t_size}, $|A_i \cap E + A_j\cap E| = |A_i\cap E|$. Let $a \in A_j \cap E$. Then 
	\begin{align*}
		|A_i\cap E| =  |A_i \cap E + A_j\cap E| & \geq |A_i \cap E + a| = |A_i \cap E|. 
	\end{align*}
	
This implies that $A_i \cap E + A_j\cap E = A_i \cap E + a$. Hence $A_j\cap E \subseteq \Stab(A_i\cap E) + a$. Therefore
	\begin{align*}
		|\Stab(A_i\cap E)| \geq |A_j\cap E|& = |A_i\cap E| \geq |\Stab(A_i\cap E)|.
	\end{align*}
	
Hence, it follows that $|A_i \cap E| = |\Stab(A_i \cap E)|$. Moreover, $A_j\cap E = \Stab(A_i\cap E) + a$. Taking the stabiliser on both sides, it follows that $\Stab(A_i \cap E)=\Stab(A_j \cap E)$. 
\end{subproof}

\ref{2t_stab} implies that the three stabilisers coincide on a flat, and that for each $i=1,2,3$, $A_i \cap E$ is a coset of that flat. 

Now, we may run the same argument, with the role of the $A$ cosets swapped with that of the $B$ cosets. This gives the following.

\begin{claim}\label{stab_property}\
\begin{itemize}
	\item $\Stab(A_1 \cap E) = \Stab(A_2 \cap E) = \Stab(A_3 \cap E)$.
	\item $\Stab(B_1 \cap E) = \Stab(B_2 \cap E) = \Stab(B_3 \cap E)$.
	\item For any $X \in \{A_1, A_2, A_3, B_1, B_2, B_3\}$, $|X \cap E| = |\Stab(X \cap E)|$.
\end{itemize}
\end{claim}

Let $F_A$ be the flat of $W$ for which $F_A = \Stab(A_1 \cap E)$, and $F_B$ another flat of $W$ for which $F_B = \Stab(B_1 \cap E)$. 

Fix any two elements $a_1 \in A_1 \cap E$ and $a_2 \in A_2 \cap E$, and similarly $b_1 \in B_1 \cap E$ and $b_2 \in B_2 \cap E$. Then by \ref{2t_c5} and \ref{stab_property}, we see that $E \cap A_i = \cl(F_A \cup \{a_i\})$ for $i=1,2$, and $E \cap A_3$ equals the set $a_1+a_2+A_3'$ where $A_3 ' =  \cl(F_A \cup \{y_A\}) \del F_A$ for some $y_A \in F \cap E$. Similarly, $E \cap B_i =\cl(F_B \cup \{b_i\})$ for $i=1,2$, and $E \cap B_3$ equals the set $b_1+b_2+B_3'$ where $B_3' = \cl(F_B\cup \{y_B\})\del F_B$ for some $y_B \in F \cap E$.

\begin{claim}\label{affine_cover}
$A_3' \cup B_3' = F \cap E$.
\end{claim}

\begin{subproof}
Suppose not. Pick $z \in (F \cap E) \del (A_3' \cup B_3')$. Then $\{a_1, a_2, b_1, b_2, z\}$ is an induced $I_5$-restriction, a contradiction
\end{subproof}

\begin{claim}\label{both_hyperplanes}
$F_A$ and $F_B$ are hyperplanes of $W$.
\end{claim}

\begin{subproof}
First note that $F_A, F_B \neq W$; if $F_A = W$, then $M | \cl(F \cup A_1)$ or $M | \cl(F \cup B_1)$ is a strictly larger affine geometry containing $M | F$. 

Now, suppose for a contradiction that at least one of $F_A$ or $F_B$, say $F_A$ without loss of generality, has codimension at least $2$.  But then $|(F \cap E) \del (A_3' \cup B_3')| \geq  2^{\dim(F)-1} - 2^{\dim(F_A)} - 2^{\dim(F_B)} \geq 2^{\dim(F)-1} - 2^{\dim(F)-3} - 2^{\dim(F)-2} = 2^{\dim(F)-3} > 0$. This contradicts \ref{affine_cover}. Therefore both $F_A$ and $F_B$ are hyperplanes of $W$. 
\end{subproof}

By \ref{both_hyperplanes}, $F_A$ is a hyperplane, and therefore it follows that $M | \cl(F_A \cup \{y_A, a_1, a_2\})$ is an affine geometry of rank $\dim(F)+1$. 
\end{proof}

\subsection{$I_3$-freeness and $2T$-freeness}

In this subsection, we determine the smallest $I_3$-free, $2T$-free matroids. 

We use the following lemma in [\ref{nn1}]. 

\begin{lemma}[{[\ref{nn1}], Lemma 2.14}]\label{PQR}
	Let $(P, Q, R)$ be a partition of a binary projective geometry $G$ for which no triangle $T$ of $G$ satisfies $|T \cap P| \ge 1$ and $|T \cap R| = 1$. 
	Then
	\begin{itemize}
		\item $\cl(P) \subseteq P \cup Q$, and
		\item All cosets of $\cl(P)$ in $G$ are contained in $Q$ or $R$.
	\end{itemize}
\end{lemma}

We now prove the following main result of this subsection. We remark that it is possible to use the structural theorem for claw-free matroids from [\ref{nn1}] to give an alternative proof; we opt for a more direct approach involving less machinery. Here, $\PGS(t_1,t_2)$ is the direct sum of a projective geometry of dimension $t_1$ and a projective geometry of dimension $t_2$, called a \emph{PG-sum}.

\begin{lemma}\label{i3_2t_free_extremal}
Let $M = (E,G)$ be an $r$-dimensional, full-rank matroid that is $I_3$-free and $2T$-free. Then $|E| \geq 2^{r-1}$. Moreover, equality holds if and only if $M \cong \AG(r-1, 2)$ or $M$ can be obtained by a sequence of doublings of $\PGS(1,t)$ for some $t \leq r$. 
\end{lemma}

\begin{proof}
The result is trivially true when $r=1$. Let $M$ be a minimum counterexample, on $\dim(M)$.

We first claim that for every proper flat $F \subseteq G$, there exists a mixed coset $A$ of $F$, meaning $0 < |A \cap E| < |A|$. Otherwise we may contract $F$, and $M/F$ is a full-rank, $I_3$-free and $2T$-free matroid. Since $\dim(M/F) < \dim(M)$, it follows that $|M / F| \geq 2^{r-\dim(F)-1}$. Each non-empty coset of $F$ contains $2^{\dim(F)}$ elements of $E$, so we have that $|E| \geq |E \cap F| + 2^{r-\dim(F)-1} \cdot 2^{\dim(F)} \geq 2^{r-1}$. Moreover, $|E| = 2^{r-1}$ if and only if $|F \cap E| = 0$ and the bound is attained for $M/F$. Hence, it follows that $M$ is the $\dim(F)$-doubling of $M/F$, and $M / F$ is either an affine geometry, or can be obtained via doublings of some $\PGS(1,t)$. Hence $M$ is an affine geometry, or $M$ can be obtained by a sequence of doublings of $\PGS(1,t)$, which contradicts minimality. Hence we may assume that there exists some mixed coset.

\begin{claim}
There exists a hyperplane $H$ of $G$ for which $M | H$ is rank-deficient.
\end{claim}

\begin{subproof}
Suppose not, so that $M | H$ is full-rank for every hyperplane $H$. By minimality, $|E \cap H| \geq 2^{r-2}$. Let $\cH$ be the set of hyperplanes of $G$. Note that any given point in $G$ is contained in precisely ${r-1 \choose r-2}_2$ hyperplanes of $G$. By a standard double counting argument, we have

	\begin{align*}
		|E| =  \frac{1}{{r-1 \choose r-2}_2} \sum_{H \in \cH} |E \cap H|
		 & \geq  \frac{{r \choose r-1}_2}{{r-1 \choose r-2}_2} \cdot 2^{r-2}\\
		& > 2 \cdot 2^{r-2} \\
		& = 2^{r-1}
	\end{align*}
	
This contradicts the minimality of $M$. 
\end{subproof}

Let $H$ be a hyperplane of $G$ for which $M|H$ is rank-deficient. Let $F$ be a hyperplane of $H$ such that $E \cap H \subseteq F$. Let $A_1$ and $A_2$ be the remaining two cosets of $F$. We may assume that $A_i \cap E \neq \varnothing$ for $i=1,2$; otherwise $M$ is rank-deficient.

\begin{claim}\label{useful_2tclaw}
If $v \in (A_i \cap E + A_i \cap E) \del E$ for any $i \in \{1,2\}$, then $v+w \in E$ for all $w \in (E \cap A_1) \cup (E \cap A_2)$.
\end{claim}

\begin{subproof}
Observe that if $v \in (A_i \cap E + A_i \cap E) \del E$, then for any $b \in E \cap A_{i-1}$ (indices taken modulo $2$), $v+b \in E$; write $v=a_1+a_2$ where $a_1,a_2 \in A_i \cap E$, then otherwise $\{a_1,a_2,b\}$ is a claw. Furthermore, since $|E \cap A_{i-1}| > 0$, there exists some $b \in E \cap A_{i-1}$, and by this observation, $v+b \in E \cap A_{i-1}$, so $v \in (A_{1-i} \cap E + A_{1-i} \cap E) \del E$. Applying the same observation, we conclude that if $v \in (A_i \cap E + A_i \cap E) \del E$ for some $i \in \{1,2\}$, then in fact $v+w \in E$ for all $w \in (E \cap A_1) \cup (E \cap A_2)$. 
\end{subproof}

\begin{claim}
$(A_1 \cap E + A_1 \cap E) \del E = (A_2 \cap E + A_2 \cap E) \del E$. 
\end{claim}

\begin{subproof}
Let $v \in (A_i \cap E + A_i \cap E) \del E$. Since $|E \cap A_{i-1}| > 0$, take $b \in E \cap A_{i-1}$. By \ref{useful_2tclaw}, $v+b \in E \cap A_{i-1}$. 
\end{subproof}

\begin{claim}\label{no_two_points_F}
$(A_1 \cap E + A_1 \cap E) \del E = \varnothing$.
\end{claim}

\begin{subproof}
Let $P= (A_1 \cap E + A_1 \cap E) \del E$, $Q = F \cap E$, $R=F \del (P \cup Q)$. We check that there are no triangles $T$ in $F$ for which $|T \cap P| \geq 1$ and $|T \cap R|=1$. Suppose for a contradiction that such a triangle $T=\{v,w,v+w\}$ exists, with $v \in P$ and $w \in R$. Fix an elements $a_1 \in E \cap A_1$. Since $v \in P$ it follows from \ref{useful_2tclaw} that $v+a_1 \in E$. Then, since $w \in R$, $a_1+w \notin E$ and $a_1+v+w \notin E$. Now, if $v+w \in E$, then $\{a_1, a_1+v,v+w\}$ is a claw, a contradiction, so $v + w \notin E$. If $v+w \in P$, then \ref{useful_2tclaw} would imply that $a_1+v+w \in E$, a contradiction. So $v+w \in R$. By Lemma \ref{PQR}, it follows that $\cl(P) \subseteq P \cup Q$, and the cosets of $\cl(P)$ in $F$ are contained in $Q$ or $R$. We now claim that the flat $F' = \cl(P)$ has no mixed cosets.

First, take a coset $B$ of $F'$ in $F$. Then either $B \subseteq Q$, in which case $B \subseteq E$, and if $B \subseteq R$, then $B \cap E = \varnothing$. It remains to show that the cosets of the form $a + \vs{F'}$ where  $a \in A_1 \cup A_2$ are not mixed. Take $a \in E \cap A_i$ for some $i \in \{1,2\}$ and we show that $a + \vs{F'} \subseteq E$. Let $b \in E \cap A_{1-i}$.

 Let $1P= P$, and let $kP = P + (k-1)P$ for $k \geq 2$. Then $\cl(P) = \cup_{k \geq 1}(kP)$. We prove by induction that for any $k \geq 1$, any $v \in kP$ satisfies $a+v, b+v \in E$. The base case $k=1$ follows from \ref{useful_2tclaw}. Now suppose that the statement is true for $(k-1)P$. Let $v \in kP$ so that there exists $w \in P$ such that $w+v \in (k-1)P$. If $v \in P$, then by \ref{useful_2tclaw} $a+v, b+v \in E$, so suppose not, so that $v \in E$. Then at least one of $a+v \in E$ and $b+v \in E$ must hold; otherwise $\{a,b,v\}$ is a claw. Assume now for a contradiction that $a+v \notin E$ and $b+v \in E$; the other case is symmetrical. Then $w = (a+v) + (a+w+v) \in (A_i \del E) + (A_i \cap E)$, which is a contradiction to \ref{useful_2tclaw} since $w \in P$. Therefore, it follows that $a+v, b+v \in E$. By induction, $a+\vs{F'} \subseteq E$, and there are no mixed cosets of $F'$ in $G$. 

By the remark from the beginning of the proof, every proper flat has to have at least one mixed coset, so this means that $F' = \cl(P)$ is empty. Hence $(A_i \cap E + A_i \cap E) \del E = \varnothing$ for $i=1,2$. 
\end{subproof}

\begin{claim}
$F \cap E = F$.
\end{claim}

\begin{subproof}
Take $v, w \in F \cap E$, and fix $a_i \in E \cap A_i$, $i=1,2$. We claim that $v+w \in E$. So suppose not. By \ref{no_two_points_F}, $a_i+v+w \notin E$ for $i=1,2$. Note that for each $i=1,2$, precisely one of $a_i+v \in E$ and $a_i+w \in E$ must hold; if none holds, then $\{a_i,v,w\}$ is a claw, and if both hold, then the triangle $\{a_i+v, a_i+w, v+w\}$ violates \ref{no_two_points_F}. Assume without loss of generality that $a_1+v \in E$ (and $a_1+w \notin E$). Now, if $a_2+w \in E$ (and hence $a_2+v \notin E$), then $M \rvert \cl(\{a_1, v, a_2, w\})$ is an induced $2T$-restriction. Hence $a_2+v \in E$ and $a_2+w \notin E$. But then $\{a_1,a_2,w\}$ is a claw, a contradiction. So $v+w \in E$. We therefore conclude that $\cl(F \cap E) \subseteq E$.

It remains to show that $\cl(F \cap E) = F$. If $\cl(F \cap E)$ is a proper subset of $F$, then by picking $a_i \in E \cap A_i$ for $i=1,2$ and observing \ref{no_two_points_F}, we see that $E \subseteq \cl((F \cap E) \cup \{a_1,a_2\})$, and so $M$ is rank-deficient, a contradiction. Therefore $\cl(F \cap E) = F$ and hence $F \cap E = F$. 
\end{subproof}

Finally, we uniquely determine the matroid $M$. 

\begin{claim}
$M \cong \PGS(1,r-1)$.
\end{claim}

\begin{subproof}
We fix $a_i \in A_i \cap E$, $i=1,2$. Take $v \in F$. Then $|E \cap \{v+a_1,v+a_2\}| \geq 1$; otherwise $\{v,a_1,a_2\}$ is a claw. Hence $|E| \geq 2+2|F| = 2^{r-1}$. By minimality, $|E| = 2^{r-1}$. This equality occurs only if for every $v \in F$, $|E \cap \{v+a_1,v+a_2\}| = 1$. Under this assumption that each $v \in F$ satisfies precisely one of $v+a_1 \in E$ or $v+a_2 \in E$, we now show that there exists $i \in \{1,2\}$ such that $v+a_i \in E$ for all $v \in F$. Suppose not, so there exist $v,w \in F$ for which $v+a_1, w+a_2 \in E$. Without loss of generality, suppose $v+w+a_1 \in E$. But then $\{a_2, v, v+w+a_1\}$ is a claw, a contradiction. 

This implies that $E = \cl(F \cup \{a_i\}) \cup \{a_{1-i}\}$, so $M \cong \PGS(1,r-1)$. 
\end{subproof}
\end{proof}

\section{Main Theorem}
We can now combine the results in the prior sections to give our main result, restated below.

\begin{theorem}
Let $M=(E,G)$ be an $r$-dimensional, full-rank, $I_5$-free, and triangle-free matroid. Then $|E| \geq 2^{\lfloor r/2 \rfloor -1 } + 2^{\lceil r/2 \rceil -1 }$. Moreover, when $r \geq 6$, equality holds if and only if $M$ is the direct sum of two affine geometries of dimension $\lfloor r/2 \rfloor$ and $\lceil r/2 \rceil$ respectively.
\end{theorem}

\begin{proof}
By Lemma \ref{extremal_examples_are_affine}, if $M$ contained an induced $C_5$-restriction, then the bound holds, and when $r \geq 6$, no such matroids attain the bound. Hence we may assume that $M$ is $C_5$-free, hence affine.

Suppose first that $M$ is $C_4$-free. We first claim that $M$ can have rank at most $5$. For a contradiction, take an independent set of six elements $x_i \in E$, $i=1,\dots,6$. By triangle-freeness, $C_4$-freeness and $C_5$-freeness, it follows that $\sum_{i \in I}x_i \notin E$ for any $I \subseteq \{1,2,3,4,5,6\}$ and $|I| = 2,3,4$, and therefore, because of $I_5$-freeness, $\sum_{i \in I}x_i \in E$ for any $|I|=5$. But then $\{x_1+x_2+x_3+x_4+x_5, x_2+x_3+x_4+x_5+x_6, x_1, x_6 \}$ is an induced $C_4$-restriction, a contradiction. It is then easy to check that the only possible full-rank matroids that are also $C_4$-free are $I_1$, $I_2$, $I_3$, $I_4$ and $C_6$, all of which satisfy the theorem.

Suppose that $M|F$ is a maximum affine geometry; note that $\dim(F) \geq 3$ since $M$ has an induced $C_4$-restriction. Let $M_F = (E_F, F')$ denote $M/F$. Note that $M_F$ is full-rank. By Lemma \ref{contraction_max_ag} and maximality, either $M_F$ is $I_3$-free, or there exists a doubled kite, say of order $k$ (so that the dimension of $F$ is $k+3$). Then by Lemma \ref{kite_too_many_elements}, $|E|$ does not attain the bound. Hence we may assume that $M_F$ is $I_3$-free. By Lemma \ref{2tfreeness} and maximality, $M_F$ is $2T$-free. By Lemma \ref{i3_2t_free_extremal}, it follows that $|E_F| \geq 2^{r-\dim(F)-1}$. If we let $\cC$ be the set of cosets of $F$, then
	\begin{align*}
		|E| &= |E \cap F| + \sum_{A \in \cC} |E \cap A| \\
		&\geq |E \cap F| + |E_F| \\
		& \geq 2^{\dim(F)-1} +  2^{r-\dim(F)-1}\\
		& \geq 2^{\lfloor r/2 \rfloor -1 } + 2^{\lceil r/2 \rceil -1 }
	\end{align*}

This proves the bound. We now determine the extremal examples. First, note that when $M$ is the direct sum of two affine geometries of dimension $\lfloor r/2 \rfloor$ and $\lceil r/2 \rceil$ respectively, then clearly $M$ is both $I_5$-free and triangle-free.

In order for equality to hold, $\dim(F) = \lfloor r/2 \rfloor $ or $\dim(F) = \lceil r/2 \rceil$, each non-empty coset of $F$ contains precisely one element of $E$, and $|E_F| = 2^{r-\dim(F)-1}$. Moreover, by Lemma \ref{2tfreeness}, $M_F \cong \AG(r-\dim(F)-1,2)$ or $M_F$ is obtained by a sequence of doublings of $\PGS(1,t)$ for some $t$. 

\textbf{Case 1}: $M_F$ is obtained by a sequence of doublings of $\PGS(1,t)$.

Note that if $t =0,1$, then doublings of $\PGS(1,t)$ are affine geometries, which will be handled in Case 2, so suppose $t > 1$. In particular, this means that $M_F$ contains $\PGS(1,2)$ as an induced restriction, corresponding to four non-empty cosets $A_1$, $A_2$, $A_3$ and $A_4$ for which $A_1+A_2=A_4$. Fix the (unique) elements $x_i \in A_i \cap E$ for $i=1,2,3$. Since $M$ is triangle-free and $C_5$-free, it follows that the (unique) element $x_4 \in A_4 \cap E$ satisfies $x_4 = x_1+x_2+y$ for some $y \in F \cap E$ (if $y \in F \del E$, then $\{x_1,x_2, x_1+x_2+y, z,y+z\}$ is an induced $C_5$-restriction for any choice of $z \in F \cap E$). Pick $z_1, z_2 \in (F \cap E) \del \{y\}$. Then $\{x_1,x_2,x_3,z_1,z_2\}$ is an induced $I_5$-restriction, a contradiction. So Case 1 does not arise.

\textbf{Case 2}: $M_F \cong \AG(r-\dim(F)-1,2)$.

Take a basis (in terms of cosets) $A_1, \dots, A_{l}$ of $M_F$ where $l = r- \dim(F)$, and from each of these cosets, take its (unique) element $x_i \in E \cap A_i$. Let $F' = \cl(x_1,\dots,x_l)$.

We claim that $M \rvert F' \cong \AG(r-\dim(F)-1,2)$. First note that $M \rvert F'$ is triangle-free since $M_F$ is triangle-free. Now suppose for a contradiction that there exist $y_1,y_2,y_3 \in F' \cap E$ such that $y_1+y_2+y_3 \notin E$. Since $\dim(F) > 2$, we may pick two elements $z_1,z_2 \in E \cap F$ for which $y_1+y_2+y_3+z_1+z_2 \notin E$, but then $\{y_1,y_2,y_3,z_1,z_2\}$ is an induced $I_5$-restriction, a contradiction. So $M \rvert F'$ is $I_3$-free. By Lemma \ref{i3triangle}, it follows that $M \rvert F' \cong \AG(r-\dim(F)-1,2)$. 

Combining these conditions, we obtain that $ |E| =  2^{\lfloor r/2 \rfloor -1 } + 2^{\lceil r/2 \rceil -1 }$ if and only if $M$ is the direct sum of two affine geometries of dimension $\lfloor r/2 \rfloor$ and $\lceil r/2 \rceil$ respectively. 
\end{proof}

\section*{References}
\newcounter{refs}
\begin{list}{[\arabic{refs}]}
{\usecounter{refs}\setlength{\leftmargin}{10mm}\setlength{\itemsep}{0mm}}

\item\label{bb}
R. C. Bose, R. C. Burton, 
A characterization of flat spaces in a finite geometry and the uniqueness of the Hamming and the MacDonald codes, 
J. Combin. Theory 1 (1966), 96--104. 

\item\label{bkknp}
M. Bonamy, F. Kardo\v{s}, T. Kelly, P. Nelson, L. Postle,
The structure of binary matroids with no induced claw or Fano plane restriction,
Advances in Combinatorics, 2019:1,17 pp.

\item\label{gs}
P. Govaerts, L. Storme,
The classification of the smallest non-trivial blocking sets in $\PG(n,2)$,
J. Combin. Theory Ser. A 113 (2006), 1543-1548.

\item\label{nn1}
P. Nelson, K. Nomoto,
The structure of claw-free binary matroids.
arXiv:1807.11543 (2018).

\item\label{nn2}
P. Nelson, K. Nomoto,
The structure of $I_4$-free, triangle-free binary matroids.
arXiv:2005.00089 (2020)

\item\label{nnorin}
P. Nelson, S. Norin,
The smallest matroids with no large independent flat.
arXiv:1909.02045 (2019)

\item \label{oxley}
J. G. Oxley, 
Matroid Theory,
Oxford University Press, New York (2011).

\end{list}

\end{document}